\documentclass[12pt]{article}
\usepackage{graphicx}
\usepackage{amsmath,amsthm,amssymb,amsfonts,euscript,enumerate}
\usepackage{amsthm}
\usepackage{color,wrapfig}
\usepackage{pxfonts}
\usepackage{verbatim}
\newtheorem{thm}{Theorem}[section]
\newtheorem{cor}[thm]{Corollary}
\newtheorem{lem}[thm]{Lemma}

\newtheorem{rmk}[thm]{Remark}

\newcommand{\1}{\partial}
\newcommand{\2}{\overline}
\newcommand{\3}{\varepsilon}

\newcommand{\R}{{\mathbb R}}

\newcommand{\Z}{{\mathbb Z}}

\newcommand{\cD}{{\mathcal D}}

\newcommand{\cX}{{\mathcal X}}

\newcommand{\ve}{\varepsilon}

\begin{document}
\title{Existence of self-similar solution of the inverse mean curvature flow}
\author{Kin Ming Hui\\
Institute of Mathematics, Academia Sinica\\
Taipei, Taiwan, R. O. C.}
\date{Aug 23, 2018}
\smallbreak \maketitle
\begin{abstract}
We will give a new proof of a recent result of P.~Daskalopoulos, G.~Huisken and J.R.~King (\cite{DH} and reference [7] of \cite{DH}) on the existence of self-similar solution of the inverse mean curvature flow which is the graph of a radially symmetric solution in $\R^n$, $n\ge 2$, of the form
$u(x,t)=e^{\lambda t}f(e^{-\lambda t} x)$ for any constants $\lambda>\frac{1}{n-1}$ and $\mu<0$ such that $f(0)=\mu$. More precisely we will give a new proof of the existence of a unique radially symmetric solution $f$ of the equation $\mbox{div}\,\left(\frac{\nabla f}{\sqrt{1+|\nabla f|^2}} \right)=\frac{1}{\lambda}\cdot\frac{\sqrt{1+|\nabla f|^2}}{x\cdot\nabla f-f}$ in $\R^n$, $f(0)=\mu$, for any $\lambda>\frac{1}{n-1}$ and $\mu<0$, which satisfies $f_r(r)>0$, $f_{rr}(r)>0$ and $rf_r(r)>f(r)$ for all $r>0$. We will also prove that $\lim_{r\to\infty}\frac{rf_r(r)}{f(r)}=\frac{\lambda (n-1)}{\lambda (n-1)-1}$.
\end{abstract}

\vskip 0.2truein

Key words: inverse mean curvature flow, self-similar solution, existence, asymptotic behaviour

AMS 2010 Mathematics Subject Classification: Primary 35K67, 35J75 Secondary 53C44

\vskip 0.2truein
\setcounter{equation}{0}
\setcounter{section}{0}

\section{Introduction}
\setcounter{equation}{0}
\setcounter{thm}{0}

Consider a family of immersions $F:M^n\times [0,T)\to\R^{n+1}$ of $n$-dimensional hypersurfaces in $\R^{n+1}$. We say that $M_t=F_t(M^n)$, $F_t(x)=F(x,t)$,  moves by the inverse mean curvature flow if 
\begin{equation*}
\frac{\1}{\1 t}F(x,t)=\frac{\nu}{H}\quad\forall x\in\R^n, 0<t<T
\end{equation*}
where $H(x,t)>0$ and $\nu$ are the mean curvature and unit exterior normal of the surface
$F_t$ at the point $F(x,t)$. Note that when $M_t$ is the graph $\2{F}(x,t)=(x.u(x,t))$ of some function $u:\R^n\times (0,T) \to \R$, $n\ge 1$, then
\begin{equation*}
\nu=\left(\frac{\nabla u}{\sqrt{1+|\nabla u|^2}},\frac{-1}{\sqrt{1+|\nabla u|^2}}\right).
\end{equation*}
Recently there has been a lot of study on the inverse mean curvature flow  for the compact case by 
 C.~Gerhardt, G.~Huisken, T.~Ilmanen, K.~Smoczyk, J.~Urbas and others  \cite{G},  \cite{HI1}, \cite{HI2}, \cite{HI3}, \cite{S}, \cite{U}. There are also a lot of progress for the non-compact case recently by B.~Allen, P.~Daskalopoulos, G.~Huisken, B.~Lambert, T.~Marquardt and J.~Scheuer \cite{A}, \cite{DH}, \cite{LS}, \cite{M1}, \cite{M2}. 

As observed by P.~Daskalopoulos and G.~Huisken in \cite{DH}, if $M_t$ is the graph $\2{F}(x,t)=(x.u(x,t))$ of some function $u:\R^n\times (0,T) \to \R$, $n\ge 1$, then $u$ satisfies 
\begin{equation}\label{imcf-graph-eqn}
u_t=-\sqrt{1+|\nabla u|^2}\left(\mbox{div}\,\left(\frac{\nabla u}{\sqrt{1+|\nabla u|^2}}\right)\right)^{-1}
\end{equation}
and if $f:\R^n\to \R$ is solution of 
\begin{equation}\label{imcf-graph-elliptic-eqn}
\mbox{div}\,\left(\frac{\nabla f}{\sqrt{1+|\nabla f|^2}}\right)=\frac{1}{\lambda}\cdot\frac{\sqrt{1+|\nabla f|^2}}{x\cdot\nabla f-f}\quad\mbox{ in }\R^n,
\end{equation}
then for any $\lambda>0$, the function
\begin{equation*}
u(x,t)=e^{\lambda t}f(e^{-\lambda t}x), \quad (x,t)\in\R^n\times\R
\end{equation*}
is a self-similar solution of \eqref{imcf-graph-eqn} in $\R^n\times \R$. In \cite{DH} and reference [7] of \cite{DH} P.~Daskalopoulos, G.~Huisken and J.R.~King also stated the existence of radially symmetric solution of \eqref{imcf-graph-elliptic-eqn} for any $n\ge 2$, $\lambda>\frac{1}{n-1}$ and $\mu:=f(0)<0$. Note that if $f$ is a radially symmetric solution of \eqref{imcf-graph-elliptic-eqn}, then $f$ satisfies
\begin{equation}\label{imcf-graph-ode}
f_{rr}+\frac{n-1}{r}\cdot(1+f_r^2)f_r-\frac{1}{\lambda}\cdot\frac{(1+f_r^2)^2}{rf_r-f}=0\quad\forall r>0
\end{equation}
and $f_r(0)=0$.
Since there is no proof of this result in \cite{DH}, in this paper I will give a detailed proof of the existence of solution of \eqref{imcf-graph-ode}. More precisely I will prove the following existence result.

\begin{thm}\label{existence_soln-thm}
For any $n\ge 2$, $\lambda>\frac{1}{n-1}$ and $\mu<0$, the equation 
\begin{equation}\label{imcf-graph-ode-initial-value-problem}
\left\{\begin{aligned}
&f_{rr}+\frac{n-1}{r}\cdot(1+f_r^2)f_r-\frac{1}{\lambda}\cdot\frac{(1+f_r^2)^2}{rf_r-f}=0\quad\forall r>0\\
&f(0)=\mu,\quad f_r(0)=0\end{aligned}\right.
\end{equation}
has a unique solution $f\in C^1([0,\infty))\cap C^2(0,\infty)$ which satisfies
\begin{equation}\label{f-structure-ineqn}
rf_r(r)>f(r)\quad\forall r\ge 0
\end{equation}
and
\begin{equation}\label{fr-+ve}
f_r(r)>0\quad\forall r>0.
\end{equation}
\end{thm}

We also obtain the following large time behavior solution of \eqref{imcf-graph-ode-initial-value-problem}.

\begin{thm}\label{asymptotic-behaviour-time-infty-thm}
Let $n\ge 2$, $\lambda>\frac{1}{n-1}$, $\mu<0$ and $f\in C^1([0,\infty))\cap C^2(0,\infty)$ be the unique solution of \eqref{imcf-graph-ode-initial-value-problem}. Then
\begin{equation}\label{growth-rate}
\mbox{$\lim_{r\to\infty}$}\frac{rf_r(r)}{f(r)}=\frac{\lambda (n-1)}{\lambda (n-1)-1}.
\end{equation}
\end{thm}

\begin{rmk}
Note that the condition $\mu<0$ is imposed to ensure the positivity of the denominator of the third term of \eqref{imcf-graph-ode-initial-value-problem}  so that one can obtain the convexity of the solution $f$ of \eqref{imcf-graph-ode-initial-value-problem} which is stated in Corollary \ref{f-to-infty-cor}.
\end{rmk}

The plan of the paper is as follows. In section 2 we will prove Theorem \ref{existence_soln-thm}. In section 3 we will prove Theorem \ref{asymptotic-behaviour-time-infty-thm}.

\section{Existence of solution}
\setcounter{equation}{0}
\setcounter{thm}{0} 

In this section we will prove Theorem \ref{existence_soln-thm}. We will first use a fixed point argument to prove the  existence of a solution of \eqref{imcf-graph-ode-initial-value-problem} in a small interval of the origin. The local solution is then extended to a global solution of \eqref{imcf-graph-ode-initial-value-problem} by a continuity argument using another fixed argument. We first start with a lemma.

\begin{lem}\label{local-existence-lem}
For any $n\ge 2$, $\lambda>0$ and $\mu<0$, there exists a constant $R_0>0$ such that the equation 
\begin{equation}\label{imcf-graph-ode-initial-value-problem2}
\left\{\begin{aligned}
&f_{rr}+\frac{n-1}{r}\cdot(1+f_r^2)f_r-\frac{1}{\lambda}\cdot\frac{(1+f_r^2)^2}{rf_r-f}=0\quad\mbox{ in }(0,R_0)\\
&f(0)=\mu,\quad f_r(0)=0\end{aligned}\right.
\end{equation}
has a unique solution $f\in C^1([0,R_0))\cap C^2(0,R_0)$ which satisfies
\begin{equation}\label{f-structure-ineqn2}
rf_r(r)-f(r)>0\quad\mbox{ in }[0,R_0).
\end{equation}
\end{lem}
\begin{proof}
Uniqueness of solution of \eqref{imcf-graph-ode-initial-value-problem2} follows from standard ODE theory. Hence we only need to prove existence of solution of \eqref{imcf-graph-ode-initial-value-problem2}. We first observe that if $f$ satisfies \eqref{imcf-graph-ode-initial-value-problem2} and \eqref{f-structure-ineqn2} for some constant $R_0>0$, then by multiplying \eqref{imcf-graph-ode-initial-value-problem2} by $r$ and integrating over $(0,r)$,
we get
\begin{equation*}
\int_0^rsf_{rr}(s)\,ds+(n-1)\int_0^r(1+f_r(s)^2)f_r(s)\,ds=\frac{1}{\lambda}\int_0^r\frac{s(1+f_r(s)^2)^2}{sf_r(s)-f(s)}\,ds\quad \forall 0<r<R_0.
\end{equation*}
Hence
\begin{equation}\label{imcf-graph-ode-integral}
rf_r(r)+(n-2)\int_0^rf_r(s)\,ds=\frac{1}{\lambda}\int_0^r\frac{s(1+f_r(s)^2)^2}{sf_r(s)-f(s)}\,ds-(n-1)\int_0^rf_r(s)^3\,ds\quad \forall 0<r<R_0.
\end{equation}
Let 
\begin{equation}\label{big-h-defn}
H(r)=\int_0^rf_r(s)\,ds
\end{equation}
and
\begin{equation}\label{e-defn}
E(r)=\frac{1}{\lambda}\int_0^r\frac{s(1+f_r(s)^2)^2}{sf_r(s)-f(s)}\,ds-(n-1)\int_0^rf_r(s)^3\,ds.
\end{equation}
Then \eqref{imcf-graph-ode-integral} is equivalent to 
\begin{equation}\label{H-E-eqn0}
rH_r(r)+(n-2)H(r)=E(r)\quad \forall 0<r<R_0.
\end{equation}
Hence
\begin{equation}\label{H-E-eqn}
H(r)=\frac{1}{r^{n-2}}\int_0^r\rho^{n-3}E(\rho)\,d\rho\quad \forall 0<r<R_0.
\end{equation}
Then by \eqref{H-E-eqn0} and \eqref{H-E-eqn},
\begin{align}\label{f'-eqn}
f_r(r)=&H_r(r)=\frac{1}{r}(E(r)-(n-2)H(r))\notag\\
=&\frac{1}{r}\left\{\frac{1}{\lambda}\int_0^r\frac{s(1+f_r(s)^2)^2}{sf_r(s)-f(s)}\,ds-(n-1)\int_0^rf_r(s)^3\,ds\right.\notag\\
&\qquad -\frac{(n-2)}{r^{n-2}}\int_0^r\rho^{n-3}\left[\frac{1}{\lambda}\int_0^{\rho}\frac{s(1+f_r(s)^2)^2}{sf_r(s)-f(s)}\,ds\right.
-(n-1)\left.\left.\int_0^{\rho}f_r(s)^3\,ds\right]\,d\rho\right\}
\end{align}
which suggests one to use a fixed point argument to prove existence of solution of \eqref{imcf-graph-ode-initial-value-problem2}. 

Let $0<\ve<1$. We now define the Banach space 
$$
\cX_\ve:=\left\{(g,h): g, h\in C\left( [0,\ve]; \R\right) \,\,\mbox{ such that }\,\, s^{-1/2}h(s)\in L^{\infty}(0,\ve) \right\}
$$ 
with a norm given by
$$||(g,h)||_{\cX_\ve}=\max\left\{\|g\|_{L^\infty([0, \ve])} ,\|s^{-1/2}h(s)\|_{L^{\infty}(0,\ve)}\right\}.$$  
For any $(g,h)\in \cX_\ve,$ we define  
$$\Phi(g,h):=\left(\Phi_1(g,h),\Phi_2(g,h)\right),$$ 
where for $0<r\leq\ve,$
\begin{equation}\label{eq-existence-contraction-map}
\left\{\begin{aligned}
&\Phi_1(g,h)(r):=\mu+\int_0^r h(s)\,ds,\\
&\Phi_2(g,h)(r):=\frac{1}{r}\left\{E(g,h)(r)-\frac{(n-2)}{r^{n-2}}\int_0^r\rho^{n-3}E(g,h)(\rho)\,d\rho\right\}
\end{aligned}\right.
\end{equation}
with
\begin{equation*}
E(g,h)(r)=\frac{1}{\lambda}\int_0^r\frac{s(1+h(s)^2)^2}{sh(s)-g(s)}\,ds-(n-1)\int_0^rh(s)^3\,ds.
\end{equation*}
For any $0<\eta\le |\mu|/4$, let $$\cD_{\ve,\eta}:=\left\{ (g,h)\in \cX_\ve:  ||(g,h)-(\mu,0)||_{\cX_{\ve}}\leq \eta\right\}.$$
Note that $\cD_{\ve,\eta}$  is a closed subspace of $\cX_\ve$. We will show that if $\ve\in(0,1)$ is sufficiently small, the map $(g,h)\mapsto\Phi(g,h)$ will have  a unique  fixed point in $\cD_{\ve,\eta}$.

We first  prove that $\Phi(\cD_{\ve,\eta})\subset \cD_{\ve,\eta}$ if $\ve\in(0,1)$ is sufficiently  small. Let $(g,h)\in \cD_{\ve,\eta}$. Then 
\begin{equation*}
|s^{-1/2}h(s)|\le\eta\le |\mu|/4\quad\mbox{ and } \quad|g(s)-\mu|\le |\mu|/4\quad\forall 0<s\le\3.
\end{equation*}
Hence
\begin{equation}\label{h-g-bd}
|h(s)|\le \eta s^{1/2}\le (|\mu|/4)s^{1/2}\quad\mbox{ and } \quad \frac{5\mu}{4}\le g(s)\le\frac{3\mu}{4}\quad\forall 0\le s\le\3.
\end{equation} 
Thus
\begin{equation}\label{h-g-lower-bd}
sh(s)-g(s)\ge\frac{3|\mu|}{4}-\frac{|\mu|}{4}=\frac{|\mu|}{2}>0\quad\forall 0\le s\le\3.
\end{equation} 
Then
\begin{equation}\label{phi1-bd}
|\Phi_1(g,h)(r)-\mu|\le\int_0^r |h(s)|\,ds\le\eta\ve\le\eta\quad\forall 0\le r\le\3.
\end{equation}
Now by  \eqref{h-g-bd} and \eqref{h-g-lower-bd},
\begin{align}\label{E-upper-bd}
\left|E(g,h)(r)\right|\le&\left|\frac{1}{\lambda}\int_0^r\frac{s(1+h(s)^2)^2}{sh(s)-g(s)}\,ds\right|+(n-1)\left|\int_0^rh(s)^3\,ds\right|\notag\\
\le&\frac{2\left(1+(|\mu|^2/16)\right)^2}{\lambda|\mu|}\int_0^rs\,ds+(n-1)\left(\frac{|\mu|}{4}\right)^3\int_0^rs^{3/2}\,ds\notag\\
\le&c_1(r^2+r^{5/2})\quad\forall 0\le r\le\3
\end{align}
where
\begin{equation*}
c_1=\max\left(\frac{\left(1+(|\mu|^2/16)\right)^2}{\lambda|\mu|},\frac{2(n-1)}{5}\left(\frac{|\mu|}{4}\right)^3\right).
\end{equation*}
Then by \eqref{E-upper-bd},
\begin{align}\label{E-integral-upper-bd}
\frac{(n-2)}{r^{n-2}}\int_0^r\rho^{n-3}|E(g,h)(\rho)|\,d\rho
\le&\frac{(n-2)c_1}{r^{n-2}}\int_0^r\rho^{n-3}(\rho^2+\rho^{5/2})\,d\rho\notag\\
\le&(n-2)c_1(r^2+r^{5/2})\quad\forall 0<r\le\3.
\end{align}
By \eqref{eq-existence-contraction-map}, \eqref{E-upper-bd} and \eqref{E-integral-upper-bd},
\begin{equation}\label{phi2-bd}
\left|r^{-1/2}\Phi_2(g,h)(r)\right|\le (n-1)c_1 (r^{1/2}+r)\le 2(n-1)c_1r^{1/2}\le\eta\quad\forall 0<r\le\3
\end{equation}
if $0<\ve\le\ve_1$ where
$$
\ve_1=\min \left(1,\frac{\eta^2}{4(n-1)^2c_1^2}\right).
$$
Thus  by \eqref{phi1-bd} and \eqref{phi2-bd}, $\Phi(\cD_{\ve,\eta})\subset \cD_{\ve,\eta}$ for any  $0<\ve\le\ve_1$.

We now let $0<\ve\le\ve_1$. Let $(g_1,h_1),(g_2,h_2)\in \cD_{\ve,\eta}$ and $\delta:=||(g_1,h_1)-(g_2,h_2)||_{\cX_\ve}$. Then
\begin{equation}\label{h12-g12-difference}
\left\{\begin{aligned}
&s^{-1/2}|h_1(s)-h_2(s)|\le\delta\quad\,\,\forall 0<s\le\3\\
&|g_1(s)-g_2(s)|\le\delta\qquad\quad\forall 0<s\le\3.
\end{aligned}\right.
\end{equation}
By \eqref{h-g-bd} and \eqref{h-g-lower-bd},
\begin{equation}\label{hi-gi-bd}
|h_i(s)|\le \eta s^{1/2}\le (|\mu|/4)s^{1/2}, \, \frac{5\mu}{4}\le g_i(s)\le\frac{3\mu}{4}\,\mbox{ and }\, sh_i(s)-g_i(s)\ge\frac{|\mu|}{2}>0\,\,\forall 0\le s\le\3, i=1,2.
\end{equation}
Now by \eqref{h12-g12-difference},
\begin{equation}\label{phi1-contraction}
|\Phi_1(g_1,h_1)(r)-\Phi_1(g_2,h_2)(r)|\le\int_0^r |h_1(s)-h_2(s)|\,ds\le\delta\int_0^r s^{1/2}\,ds\le\frac{2\ve^{3/2}}{3}\delta \le\frac{2}{3}\delta\quad\forall 0\le r\le\3
\end{equation}
and
\begin{align}\label{phi2-difference}
&|\Phi_2(g_1,h_1)(r)-\Phi_2(g_2,h_2)(r)|\notag\\
\le&\frac{1}{r}\left\{|E(g_1,h_1)(r)-E(g_2,h_2)(r)|+\frac{(n-2)}{r^{n-2}}\int_0^r\rho^{n-3}|E(g_1,h_1)(\rho)-E(g_2,h_2)(\rho)|\,d\rho\right\}\quad\forall 0< r\le\3.
\end{align}
By \eqref{h12-g12-difference} and \eqref{hi-gi-bd},
\begin{align}\label{h-g-bd20}
&\left|\frac{(1+h_1(s)^2)^2}{sh_1(s)-g_1(s)}-\frac{(1+h_2(s)^2)^2}{sh_2(s)-g_2(s)}\right|\notag\\
\le &4\frac{\left|(1+h_1(s)^2)^2(sh_2(s)-g_2(s))-(1+h_2(s)^2)^2(sh_1(s)-g_1(s))\right|}{|\mu|^2}\quad\forall 0\le r\le\3
\end{align}
and
\begin{align}\label{h-g-bd21}
&\left|(1+h_1(s)^2)^2(sh_2(s)-g_2(s))-(1+h_2(s)^2)^2(sh_1(s)-g_1(s))\right|\notag\\
\le&\left|(1+h_1(s)^2)^2-(1+h_2(s)^2)^2\right||sh_2(s)-g_2(s)|+(1+h_2(s)^2)^2|sh_2(s)-g_2(s)-sh_1(s)+g_1(s)|\notag\\
\le&|h_1(s)-h_2(s)||h_1(s)+h_2(s)|\left|2+h_1(s)^2+h_2(s)^2\right|(|sh_2(s)|+|g_2(s)|)\notag\\
&\qquad+(1+h_2(s)^2)^2(s|h_2(s)-h_1(s)|+|g_2(s)-g_1(s)|)\notag\\
\le &\delta s^{1/2}\cdot 2\eta(2+2\eta^2)\left(\frac{|\mu|}{4}+\frac{5|\mu|}{4}\right)+(1+\eta^2)^2(s^{3/2}+1)\delta\notag\\
\le&c_2\delta\quad\forall 0\le s\le\3
\end{align}
where
\begin{equation*}
c_2=6\eta (1+\eta^2)|\mu|+2(1+\eta^2)^2.
\end{equation*}
By \eqref{h-g-bd20} and \eqref{h-g-bd21},
\begin{equation}\label{integral-ineqn1}
\int_0^r\left|\frac{s(1+h_1^2)^2}{sh_1(s)-g_1(s)}-\frac{s(1+h_2^2)^2}{sh_2(s)-g_2(s)}\right|\,ds\le\frac{2c_2r^2}{|\mu|^2}\delta\quad\forall 0\le r\le\3
\end{equation}
and by \eqref{h12-g12-difference} and \eqref{hi-gi-bd},
\begin{align}\label{integral-ineqn2}
\int_0^r|h_1(s)^3-h_2(s)^3|\,ds\le&\int_0^r|h_1(s)-h_2(s)||h_1(s)^2+h_1(s)h_2(s)+h_2(s)^2|\,ds\notag\\
\le&3\eta^2\delta\int_0^rs^{3/2}\,ds\notag\\
\le&\frac{6\eta^2r^{5/2}}{5}\delta\quad\forall 0\le r\le\3.
\end{align}
By \eqref{integral-ineqn1} and \eqref{integral-ineqn2},
\begin{equation}\label{E-difference}
|E(g_1,h_1)(r)-E(g_2,h_2)(r)|\le c_3(r^2+r^{5/2})\delta\quad\forall 0\le r\le\3
\end{equation}
where 
\begin{equation*}
c_3=\max\left(\frac{2c_2}{|\mu|^2\lambda},\frac{6(n-1)\eta^2}{5}\right).
\end{equation*}
Hence
\begin{align}\label{E-difference-integral}
\frac{(n-2)}{r^{n-2}}\int_0^r\rho^{n-3}|E(g_1,h_1)(\rho)-E(g_2,h_2)(\rho)|\,d\rho
\le&\frac{(n-2)c_3\delta}{r^{n-2}}\int_0^r\rho^{n-3}(\rho^2+\rho^{5/2})\,d\rho\notag\\
\le&(n-2)c_3(r^2+r^{5/2})\delta
\quad\forall 0\le r\le\3.
\end{align}
By \eqref{phi2-difference}, \eqref{E-difference} and \eqref{E-difference-integral},
\begin{equation}\label{phi2-contraction}
r^{-1/2}|\Phi_2(g_1,h_1)(r)-\Phi_2(g_2,h_2)(r)|\le(n-1)c_3(r^{1/2}+r)\delta\le 2(n-1)c_3r^{1/2}\delta\quad\forall 0<r\le\3.
\end{equation}
We now let 
\begin{equation*}
\ve_2=\min\left(\ve_1,\frac{1}{9(n-1)^2c_3^2}\right)
\end{equation*}
and $0<\ve\le\ve_2$.
Then by \eqref{phi1-contraction} and \eqref{phi2-contraction},
\begin{equation*}
\|\Phi(g_1,h_1)-\Phi(g_2,h_2)\|_{\cX_\ve}\le\frac{2}{3}\|(g_1,h_1)-(g_2,h_2)\|_{\cX_\ve}\quad\forall (g_1,h_1),(g_2,h_2)\in \cD_{\ve,\eta}.
\end{equation*}
Hence $\Phi$ is a contraction map on $\cD_{\ve,\eta}$. Then by the Banach fixed point theorem the map $\Phi$ has a unique fixed point. Let $(g,h)\in \cD_{\ve,\eta}$ be the unique fixed point of the map $\Phi$. Then
\begin{equation*}
\Phi(g,h)=(g,h).
\end{equation*}
Hence 
\begin{equation*}
g(r)=\mu+\int_0^r h(s)\,ds\quad\forall 0<r<\3\quad
\mbox{ and }\quad g(0)=\mu
\end{equation*}
which implies
\begin{equation}\label{g-eqn}
g_r(r)=h(r)\quad\forall 0<r<\3\quad
\mbox{ and }\quad g(0)=\mu
\end{equation}
and
\begin{equation*}
h(r)=\frac{1}{r}\left\{E(g,h)(r)-\frac{(n-2)}{r^{n-2}}\int_0^r\rho^{n-3}E(g,h)(\rho)\,d\rho
\right\}\quad\forall 0<r<\3.
\end{equation*}
Thus
\begin{equation}\label{h-formula}
r^{n-1}h(r)=r^{n-2}E(g,h)(r)-(n-2)\int_0^r\rho^{n-3}E(g,h)(\rho)\,d\rho \quad\forall 0<r<\3.
\end{equation}
Differentiating \eqref{h-formula} with respect to $r$, $\forall 0<r<\3$,
\begin{equation*}
(n-1)r^{n-2}h(r)+r^{n-1}h_r(r)=r^{n-2}\frac{\1}{\1 r} E(g,h)(r)=r^{n-2}\left\{\frac{1}{\lambda}\frac{r(1+h(r)^2)^2}{rh(r)-g(r)}-(n-1)h(r)^3\right\}.
\end{equation*}
Hence
\begin{equation}\label{h-eqn}
h_r(r)+(n-1)\frac{(h(r)+h(r)^3)}{r}=\frac{1}{\lambda}\frac{(1+h(r)^2)^2}{rh(r)-g(r)}\quad\forall 0<r<\3.
\end{equation}
By \eqref{h-g-lower-bd}, \eqref{g-eqn}, \eqref{h-formula} and \eqref{h-eqn}, $g\in C^1([0,\ve))\cap C^2(0,\ve)$ satisfies \eqref{imcf-graph-ode-initial-value-problem2} and \eqref{f-structure-ineqn2} with $R_0=\ve$ and the lemma follows.

\end{proof}

\begin{lem}\label{local-existence-extension-lem}
Let $n\ge 2$, $\lambda>0$, $r_0'\ge r_1\ge r_0>0$, $a_1>0$ and $a_0, b_0\in\R$, $|a_0|, |b_0|\le M$ for some constant $M>0$ be such that
\begin{equation}\label{a0-b0-positivity-relation}
r_1b_0-a_0\ge a_1.
\end{equation}
Then there exists a constant $\delta_1>0$ depending on $a_1$, $r_0$, $r_0'$ and $M$, but is independent of $r_1$ such that there exists a unique solution $f\in C^2([r_1,r_1+\delta_1))$ of 
\begin{equation}\label{imcf-graph-ode-bdary-value-problem}
\left\{\begin{aligned}
&f_{rr}+\frac{n-1}{r}\cdot(1+f_r^2)f_r-\frac{1}{\lambda}\cdot\frac{(1+f_r^2)^2}{rf_r-f}=0\quad\mbox{ in }(r_1,r_1+\delta_1)\\
&f(r_1)=a_0,\quad f_r(r_1)=b_0
\end{aligned}\right.
\end{equation}
which satisfies
\begin{equation}\label{f-structure-ineqn10}
rf_r(r)>f(r)\quad\forall r\in [r_1,r_1+\delta_1).
\end{equation}
\end{lem}
\begin{proof}
Uniqueness of solution of \eqref{imcf-graph-ode-bdary-value-problem} follows from standard ODE theory. Hence we only need to prove existence of solution of \eqref{imcf-graph-ode-bdary-value-problem}. We first observe that if $f$ satisfies \eqref{imcf-graph-ode-bdary-value-problem} and \eqref{f-structure-ineqn10} for some constant $\delta_1>0$, then by multiplying \eqref{imcf-graph-ode-bdary-value-problem} by $r$ and integrating over $(r_1,r)$,
we get $\forall r_1<r<r_1+\delta_1$,
\begin{equation*}
\int_{r_1}^rsf_{rr}(s)\,ds+(n-1)\int_{r_1}^r(1+f_r(s)^2)f_r(s)\,ds=\frac{1}{\lambda}\int_{r_1}^r\frac{s(1+f_r(s)^2)^2}{sf_r(s)-f(s)}\,ds.
\end{equation*}
Hence $\forall r_1<r<r_1+\delta_1$,
\begin{equation*}
rf_r(r)-r_1b_0+(n-2)\int_{r_1}^rf_r(s)\,ds=\frac{1}{\lambda}\int_{r_1}^r\frac{s(1+f_r(s)^2)^2}{sf_r(s)-f(s)}\,ds-(n-1)\int_{r_1}^rf_r(s)^3\,ds.
\end{equation*}
Thus $\forall r_1<r<r_1+\delta_1$,
\begin{equation*}
f_r(r)=\frac{1}{r}\left\{\frac{1}{\lambda}\int_{r_1}^r\frac{s(1+f_r(s)^2)^2}{sf_r(s)-f(s)}\,ds-(n-1)\int_{r_1}^rf_r(s)^3\,ds-(n-2)\int_{r_1}^rf_r(s)\,ds\right\}+\frac{r_1}{r}b_0
\end{equation*}
which suggests one to use a fixed point argument to prove existence of solution of \eqref{imcf-graph-ode-bdary-value-problem}. 

Let $\ve_1=\min\left(\frac{1}{3},\frac{a_1}{4(M+r_0'+1)}\right)$ and $0<\ve\le\ve_1$. We now define the Banach space 
$$
\cX_\ve':=\left\{(g,h): g, h\in C\left( [r_1,r_1+\ve]; \R\right)\right\}
$$ 
with a norm given by
$$||(g,h)||_{\cX_\ve'}=\max\left\{\|g\|_{L^\infty(r_1,r_1+\ve)} ,\|h(s)\|_{L^{\infty}(r_1,r_1+\ve)}\right\}.$$  
For any $(g,h)\in \cX_\ve',$ we define  
$$\Phi(g,h):=\left(\Phi_1(g,h),\Phi_2(g,h)\right),$$ 
where for $r_1<r<r_1+\ve,$
\begin{equation}\label{eq-extension-existence-contraction-map}
\left\{
\begin{aligned}
&\Phi_1(g,h)(r):=a_0+\int_{r_1}^r h(s)\,ds,\\
&\Phi_2(g,h)(r):=\frac{1}{r}\left\{\frac{1}{\lambda}\int_{r_1}^r\frac{s(1+h(s)^2)^2}{sh(s)-g(s)}\,ds-(n-1)\int_{r_1}^rh(s)^3\,ds-(n-2)\int_{r_1}^rh(s)\,ds\right\}+\frac{r_1}{r}b_0.
\end{aligned}\right.
\end{equation}
For any $0<\eta\le\ve_1$, let 
$$\cD_{\ve,\eta}':=\left\{ (g,h)\in \cX_\ve':  ||(g,h)-(a_0,b_0)||_{\cX_{\ve}'}\leq \eta\right\}.
$$
Note that $\cD_{\ve,\eta}'$  is a closed subspace of $\cX_\ve'$. We will show that if $\ve\in(0,\ve_2)$ is sufficiently  small where $\ve_2=\min (\ve_1,\eta/(M+1))$, the map $(g,h)\mapsto\Phi(g,h)$ will have  a unique  fixed point in $\cD_{\ve,\eta}'$.

We first  prove that $\Phi(\cD_{\ve,\eta}')\subset \cD_{\ve,\eta}'$ if $\ve\in (0,\ve_2)$ is sufficiently  small. Let $(g,h)\in \cD_{\ve,\eta}'$. Then 
\begin{equation*}
|h(s)-b_0|\le\eta\quad\mbox{ and } \quad|g(s)-a_0|\le \eta\quad\forall  r_1<s<r_1+\3.
\end{equation*}
Hence
\begin{equation}\label{h-g-bd4}
|h(s)|\le |b_0|+1\quad\mbox{ and } \quad |g(s)|\le |a_0|+1\quad\forall r_1< s< r_1+\3.
\end{equation}
Thus
\begin{align}\label{h-g-lower-bd5}
sh(s)-g(s)=&r_1b_0-a_0+(s-r_1)h(s)+r_1(h(s)-b_0)+(a_0-g(s))\notag\\
\ge&a_1-(1+|b_0|)\ve-r_1\eta-\eta\notag\\
\ge&a_1-\frac{a_1}{4}-\frac{a_1}{4}\notag\\
\ge&\frac{a_1}{2}>0\quad\forall  r_1\le s\le r_1+\3
\end{align}
and
\begin{equation}\label{phi1-contraction-map-5}
|\Phi_1(g,h)(r)-a_0|\le\int_{r_1}^r |h(s)|\,ds\le (1+|b_0|)\ve\le\eta\quad\forall  r_1\le r\le r_1+\3.
\end{equation}
Now by  \eqref{h-g-bd4} and \eqref{h-g-lower-bd5}, 
\begin{align}\label{phi2-contraction-map-5}
&\left|\Phi_2(g,h)(r)-b_0\right|\notag\\
\le&\left|\frac{1}{\lambda r_0}\int_{r_1}^r\frac{s(1+h(s)^2)^2}{sh(s)-g(s)}\,ds\right|+\frac{(n-1)}{r_0}\left|\int_{r_1}^rh(s)^3\,ds\right|+\frac{(n-2)}{r_0}\left|\int_{r_1}^rh(s)\,ds\right|+\frac{|r_1-r|}{|r|}|b_0|\notag\\
\le&\frac{\left(1+(1+|b_0|)^2)\right)^2}{a_1r_0\lambda}\left|r^2-r_1^2\right|+\frac{(n-1)(1+|b_0|)^3}{r_0}|r-r_1|+\frac{(n-2)(1+|b_0|)}{r_0}|r-r_1|+\frac{|r_1-r|}{r_0}|b_0|\notag\\
\le&a_2\ve\quad\forall  r_1\le r\le  r_1+\3
\end{align}
where
\begin{equation*}
a_2:=\frac{\left(1+(M+1)^2)\right)^2}{a_1r_0\lambda}(2r_0'+1)+\frac{(n-1)(M+1)^3}{r_0}+\frac{(n-2)(M+1)}{r_0}+\frac{M}{r_0}.
\end{equation*}
Let $\ve_3=\min (\ve_2,\eta/a_2)$ and $0<\3\le\ve_3$. Then by \eqref{phi2-contraction-map-5},
\begin{equation}\label{phi2-contraction-map-6}
\left|\Phi_2(g,h)(r)-b_0\right|\le\eta\quad \forall  r_1\le r\le  r_1+\3.
\end{equation}
By \eqref{phi1-contraction-map-5} and \eqref{phi2-contraction-map-6}, $\Phi(\cD_{\ve,\eta}')\subset \cD_{\ve,\eta}'$ for all $0<\3\le\ve_3$.

We now let $0<\ve\le\ve_3$. Let $(g_1,h_1),(g_2,h_2)\in \cD_{\ve,\eta}'$ and $\delta:=||(g_1,h_1)-(g_2,h_2)||_{\cX_\ve'}$. Then
\begin{equation}\label{h-g-diff50}
\left\{\begin{aligned}
&|h_1(s)-h_2(s)|\le\delta\quad\forall r_1<s<r_1+\3\\
&|g_1(s)-g_2(s)|\le\delta\quad\forall r_1<s<r_1+\3
\end{aligned}\right.
\end{equation}
and
\begin{equation*}
|h_i(s)-b_0|\le\eta\quad\mbox{ and } \quad|g_i(s)-a_0|\le \eta\quad\forall  r_1<s<r_1+\3, i=1,2
\end{equation*}
Hence
\begin{equation}
|h_i(s)|\le |b_0|+1\quad\mbox{ and } \quad |g_i(s)|\le |a_0|+1\quad\forall r_1< s< r_1+\3, i=1,2.\label{h-g-bd4'}
\end{equation}
Thus
\begin{equation}\label{phi1-contraction-map-10}
|\Phi_1(g_1,h_1)(r)-\Phi_1(g_2,h_2)(r)|\le\int_{r_1}^r |h_1(s)-h_2(s)|\,ds\le\ve\delta
\le\frac{\delta}{3}\quad \forall  r_1\le r\le  r_1+\3
\end{equation}
and
\begin{align}\label{phi2-contraction-map-10}
&|\Phi_2(g_1,h_1)(r)-\Phi_2(g_2,h_2)(r)|\notag\\
\le&\frac{1}{r_0\lambda}\int_{r_1}^r\left|\frac{(1+h_1(s)^2)^2}{sh_1(s)-g_1(s)}-\frac{(1+h_2(s)^2)^2}{sh_2(s)-g_2(s)}\right|s\,ds+\frac{(n-1)}{r_0}\int_{r_1}^r\left|h_1(s)^3-h_2(s)^3\right|\,ds\notag\\
&\qquad +\frac{(n-2)}{r_0}\int_{r_1}^r|h_1(s)-h_2(s)|\,ds\quad \forall  r_1\le r\le  r_1+\3.
\end{align}
Now by \eqref{h-g-lower-bd5},
\begin{equation}\label{hg12-diff-lower-bd20}
sh_i(s)-g_i(s)\ge\frac{a_1}{2}>0\quad\forall  r_1\le s\le r_1+\3,i=1,2.
\end{equation}
Hence by \eqref{hg12-diff-lower-bd20},
\begin{align}\label{h-g-ratio-difference}
&\left|\frac{(1+h_1(s)^2)^2}{sh_1(s)-g_1(s)}-\frac{(1+h_2(s)^2)^2}{sh_2(s)-g_2(s)}\right|\notag\\
\le &4\frac{\left|(1+h_1(s)^2)^2(sh_2(s)-g_2(s))-(1+h_2(s)^2)^2(sh_1(s)-g_1(s))\right|}{a_1^2}\quad\forall  r_1\le s\le r_1+\3.
\end{align}
By \eqref{h-g-diff50} and \eqref{h-g-bd4'},
\begin{align}\label{polynomial-difference10}
&\left|(1+h_1(s)^2)^2(sh_2(s)-g_2(s))-(1+h_2(s)^2)^2(sh_1(s)-g_1(s))\right|\notag\\
\le&\left|(1+h_1(s)^2)^2-(1+h_2(s)^2)^2\right||sh_2(s)-g_2(s)|+(1+h_2(s)^2)^2|sh_2(s)-g_2(s)-sh_1(s)+g_1(s)|\notag\\
\le&|h_1(s)-h_2(s)||h_1(s)+h_2(s)|\left|2+h_1(s)^2+h_2(s)^2\right|(|sh_2(s)|+|g_2(s)|)\notag\\
&\qquad+(1+h_2(s)^2)^2(s|h_2(s)-h_1(s)|+|g_2(s)-g_1(s)|)\notag\\
\le&a_3\delta\quad\forall r_1\le s\le r_1+\3
\end{align}
where
\begin{equation*}
a_3=8(M+1)^2(1+(M+1)^2)+2(1+(M+1)^2)^2.
\end{equation*}
Now let 
\begin{equation*}
\ve_4:=\min\left(\ve_3,\frac{a_1^2r_0\lambda}{18a_3(2r_0'+1)},\frac{r_0}{27(n-1)(1+M)^2}\right)
\end{equation*}
and let $0<\ve\le\ve_4$.
Then by \eqref{h-g-diff50}, \eqref{h-g-bd4'}, \eqref{h-g-ratio-difference} and \eqref{polynomial-difference10}, $\forall 0\le r\le\3$,
\begin{equation}\label{h-g-polyn-diff-integral-upper-bd}
\frac{1}{r_0\lambda}\int_{r_1}^r\left|\frac{(1+h_1(s)^2)^2}{sh_1(s)-g_1(s)}-\frac{(1+h_2(s)^2)^2}{sh_2(s)-g_2(s)}\right|s\,ds\le\frac{2a_3\delta}{a_1^2r_0\lambda}|r^2-r_1^2|\le\frac{2a_3(2r_0'+1)\ve}{a_1^2r_0\lambda}\delta\le\frac{\delta}{9},
\end{equation}
\begin{align}\label{h-cubic-diff-integral-upper-bd}
\frac{(n-1)}{r_0}\int_{r_1}^r\left|h_1(s)^3-h_2(s)^3\right|\,ds=&\frac{(n-1)}{r_0}\int_{r_1}^r|h_1(s)-h_2(s)|\left|h_1(s)^2+h_1(s)h_2(s)+h_2(s)^2\right|\,ds\notag\\
\le&\frac{3(n-1)(1+M)^2\ve}{r_0}\delta\le\frac{\delta}{9}\quad\forall r_1\le s\le r_1+\3,
\end{align}
and
\begin{equation}\label{h-diff-integral-upper-bd}
\frac{(n-2)}{r_0}\int_{r_1}^r|h_1(s)-h_2(s)|\,ds\le\frac{(n-2)\ve}{r_0}\delta\le\frac{\delta}{9}\quad\forall r_1\le s\le r_1+\3.
\end{equation}
By \eqref{phi2-contraction-map-10}, \eqref{h-g-polyn-diff-integral-upper-bd}, \eqref{h-cubic-diff-integral-upper-bd} and \eqref{h-diff-integral-upper-bd},
\begin{equation}\label{phi2-contraction-map-11}
|\Phi_2(g_1,h_1)(r)-\Phi_2(g_2,h_2)(r)|\le\frac{\delta}{3}\quad \forall  r_1\le r\le  r_1+\3.
\end{equation}
By \eqref{phi1-contraction-map-10} and \eqref{phi2-contraction-map-11},
\begin{equation*}
\|\Phi(g_1,h_1)-\Phi(g_2,h_2)\|_{\cX_\ve'}\le\frac{1}{3}\|(g_1,h_1)-(g_2,h_2)\|_{\cX_\ve'}\quad\forall (g_1,h_1),(g_2,h_2)\in \cD_{\ve,\eta}.
\end{equation*}
Hence $\Phi$ is a contraction map on $\cD_{\ve,\eta}'$. Then by the Banach fixed point theorem the map $\Phi$ has a unique fixed point. Let $(g,h)\in \cD_{\ve,\eta}'$ be the unique fixed point of the map $\Phi$. Then
\begin{equation*}
\Phi(g,h)=(g,h).
\end{equation*}
Hence 
\begin{equation*}
g(r)=a_0+\int_{r_1}^r h(s)\,ds\quad\forall r_1\le r\le r_1+\ve
\end{equation*}
which implies
\begin{equation}\label{g-eqn10}
g_r(r)=h(r)\quad\forall r_1\le r\le r_1+\ve\quad
\mbox{ and }\quad g(r_1)=a_0
\end{equation}
and $\forall r_1\le r\le r_1+\ve$, 
\begin{equation*}
h(r)=\frac{1}{r}\left\{\frac{1}{\lambda}\int_{r_1}^r\frac{s(1+h(s)^2)^2}{sh(s)-g(s)}\,ds-(n-1)\int_{r_1}^rh(s)^3\,ds-(n-2)\int_{r_1}^rh(s)\,ds\right\}+\frac{r_1}{r}b_0.
\end{equation*}
Thus
\begin{equation}\label{h-formula10}
\left\{\begin{aligned}
&rh(r)=\frac{1}{\lambda}\int_{r_1}^r\frac{s(1+h(s)^2)^2}{sh(s)-g(s)}\,ds-(n-1)\int_{r_1}^rh(s)^3\,ds-(n-2)\int_{r_1}^rh(s)\,ds+r_1b_0\\
&h(r_1)=b_0\end{aligned}\right.
\end{equation}
Differentiating \eqref{h-formula10} with respect to $r$,
\begin{equation}\label{h-eqn10}
rh_r(r)+h(r)=\frac{r(1+h(r)^2)^2}{\lambda(rh(r)-g(r))}-(n-1)h(r)^3-(n-2)h(r)\quad\forall r_1\le r\le r_1+\ve.
\end{equation}
By \eqref{h-g-lower-bd5}, \eqref{g-eqn10}, \eqref{h-formula10} and \eqref{h-eqn10}, $g\in C^2([r_1,r_1+\ve))$ satisfies \eqref{imcf-graph-ode-bdary-value-problem} and \eqref{f-structure-ineqn10} with $\delta_1=\ve$ and the lemma follows.
\end{proof}

\begin{lem}\label{f-monotone-lemma}
Let $n\ge 2$, $\lambda>0$, $\mu<0$ and $R_0>0$. Suppose $f\in C^1([0,R_0))\cap C^2(0,R_0)$ is the solution of \eqref{imcf-graph-ode-initial-value-problem2} which satisfies \eqref{f-structure-ineqn2}. Then 
\begin{equation}\label{f-rr}
\mbox{$\lim_{r\to 0}$} f_{rr}(r)=\frac{1}{n\lambda|\mu|}
\end{equation}
and
\begin{equation}\label{f-derivative-+ve}
f_r(r)=\frac{1}{\lambda h(r)}\int_0^r\frac{h(s)(1+f_r(s)^2)^2}{sf_r(s)-f(s)}\,ds>0\quad\forall 0<r<R_0
\end{equation}
where 
\begin{equation}\label{h-defn}
h(r)=r^{n-1}\mbox{exp}\left((n-1)\int_0^rs^{-1}f_r(s)^2\,ds\right)
\end{equation}
and there exists a constant $\delta_2>0$ such that
\begin{equation}\label{f-structure-ineqn5}
rf_r(r)-f(r)\ge\delta_2\quad\mbox{ in }[0,R_0).
\end{equation}
\end{lem}
\begin{proof}
Let $H(r)$ and $E(r)$ be given by \eqref{big-h-defn} and \eqref{e-defn}. In order to prove \eqref{f-rr} we first observe that by the proof of Lemma \ref{local-existence-lem} and \eqref{phi2-bd}, \eqref{f'-eqn} holds and there exist constants $0<R_1<R_0$ and $C_1>0$ such that
\begin{equation}\label{f'-bd1}
\frac{|f_r(r)|}{r}\le C_1\quad\forall 0<r<R_1.
\end{equation}
By \eqref{f'-bd1} the function $h$ given by \eqref{h-defn} is well-defined.
Multiplying \eqref{imcf-graph-ode-initial-value-problem2} by $h$ and integrating over $(0,r)$, \eqref{f-derivative-+ve} follows.
Let $\{r_k\}_{k=1}^{\infty}\subset (0,R_1)$ be a sequence such that $r_k\to 0$ as $k\to\infty$. By \eqref{f-derivative-+ve} and \eqref{f'-bd1} the sequence $\{r_k\}_{k=1}^{\infty}$ has a sequence which we may assume without loss of generality to be the sequence itself such that $f_r(r_k)/r_k$ converges to some point $a_0\in [0,C_1]$ as $k\to\infty$. Then by \eqref{imcf-graph-ode-initial-value-problem2}, \eqref{f'-eqn} and the l'Hospital rule,
\begin{align*}
a_0=&\mbox{$\lim_{k\to\infty}$}\frac{f_r(r_k)}{r_k}
=\mbox{$\lim_{k\to\infty}$}\frac{E(r_k)-(n-2)H(r_k)}{r_k^2}
=\mbox{$\lim_{k\to\infty}$}\frac{E_r(r_k)-(n-2)H_r(r_k)}{2r_k}\notag\\
=&\frac{1}{2}\mbox{$\lim_{k\to\infty}$}\frac{\frac{1}{\lambda}\frac{r_k(1+f_r(r_k)^2)^2}{r_kf_r(r_k)-f(r_k)}-(n-1)f_r(r_k)^3-(n-2)f_r(r_k)}{r_k}\notag\\
=&\frac{1}{2}\left(\frac{1}{\lambda|\mu|}-(n-2)a_0\right)
\end{align*} 
which implies that
\begin{equation*}
a_0=\frac{1}{n\lambda|\mu|}.
\end{equation*}
Since the sequence $\{r_k\}_{k=1}^{\infty}$ is arbitrary, 
\begin{equation}\label{f'/r-limit}
\mbox{$\lim_{r\to 0}$}\frac{f_r(r)}{r}=\frac{1}{n\lambda|\mu|}.
\end{equation}
Letting $r\to 0$ in \eqref{imcf-graph-ode-initial-value-problem2}, by \eqref{f'/r-limit} we get 
\begin{equation*}
\mbox{$\lim_{r\to 0}f_{rr}(r)$}+\frac{n-1}{n\lambda|\mu|}-\frac{1}{\lambda|\mu|}=0
\end{equation*}
and \eqref{f-rr} follows.

What is left to show is \eqref{f-structure-ineqn5}. Let $w(r)=rf_r(r)-f(r)$. By \eqref{imcf-graph-ode-initial-value-problem2} and a direct computation $w$ satisfies
\begin{equation}\label{w-derivative-eqn}
w_r(r)=r(1+f_r(r)^2)\left(\frac{1+f_r(r)^2}{\lambda w(r)}-\frac{(n-1)}{r^2}(w(r)+f(r))\right)\quad\forall 0<r<R_0.
\end{equation}
By \eqref{f-derivative-+ve}, 
\begin{equation}\label{a2-defn}
\mbox{$a_2:=\lim_{r\to R_0}f(r)\in (\mu,\infty]$}
\end{equation}
exists. We now divide the proof into 2 cases.

\noindent\textbf{Case 1}: $a_2\in (0,\infty]$

\noindent By \eqref{f-derivative-+ve} there exists  $r_1\in (R_0/2,R_0)$ such that 
\begin{equation}\label{f-lower-upper-bd}
f(r)>\min\left(\frac{a_2}{2},R_0\sqrt{(n-1)\lambda}\right)\quad\forall r_1<r<R_0.
\end{equation}
Let 
\begin{equation}\label{a4-defn}
a_3=\min_{0\le r\le r_1}w(r)
\end{equation}
and
\begin{equation}\label{a5-defn-2}
a_4=\min\left(\frac{a_2}{8(n-1)\lambda},\frac{a_3}{2},\frac{R_0}{4\sqrt{(n-1)\lambda}}\right).
\end{equation}
Then $a_3>0$ and $a_4>0$.
Suppose there exists $r_2\in (r_1,R_0)$ such that $w(r_2)<a_4$. Let $(a,b)\in (0, R_0)$ be the maximal interval containing $r_2$ such that 
\begin{equation}\label{w-upper-bd}
w(r)<a_4\quad\forall a<r<b
\end{equation} 
holds. Since $w(r_1)\ge a_3>a_4$, $a>r_1$ and $w(a)=a_4$.  By \eqref{f-lower-upper-bd}, \eqref{a5-defn-2} and \eqref{w-upper-bd}, we get
\begin{equation}\label{w-upper-bd-f-w-ineqn}
w(r)<\frac{R_0}{4\sqrt{(n-1)\lambda}}\quad\mbox{ and }\quad f(r)>4(n-1)\lambda w(r)\quad\forall a<r<b.
\end{equation}
Hence by \eqref{f-structure-ineqn2}, \eqref{f-lower-upper-bd} and \eqref{w-upper-bd-f-w-ineqn}, for any $a<r<b$ the right hand side of \eqref{w-derivative-eqn} is bounded below by
\begin{align*}
\ge&r(1+f_r(r)^2)\left(\frac{1+(f(r)/r)^2}{\lambda w(r)}-\frac{(n-1)}{r^2}(w(r)+f(r))\right)\notag\\
\ge&r(1+f_r(r)^2)\left(\frac{1+(f(r)/R_0)^2}{\lambda w(r)}-\frac{4(n-1)}{R_0^2}(w(r)+f(r))\right)\notag\\
\ge&r(1+f_r(r)^2)\left(\frac{1}{4\lambda w(r)}\left(1-\frac{16(n-1)\lambda}{R_0^2}w(r)^2\right)+
\frac{3}{4\lambda w(r)}+\frac{f(r)}{\lambda R_0^2w(r)}\left(f(r)-4(n-1)\lambda w(r)\right)\right)\\
\ge&\frac{3r_1}{4\lambda w(r)}.
\end{align*} 
Hence
\begin{align}\label{w-lower-bd5}
&w_r(r)\ge\frac{3r_1}{4\lambda w(r)}\quad\forall a<r<b\notag\\
\Rightarrow\quad&w(r)>w(a)=a_4\quad\forall a<r<b
\end{align}
which contradicts  \eqref{w-upper-bd}. Thus no such $r_2$ exists and $w(r)\ge a_4$ for all $r_1\le r<R_0$ and \eqref{f-structure-ineqn5} holds with $\delta_2=a_4$.

\noindent\textbf{Case 2}: $a_2\le 0$

Choose $r_1\in (R_0/2,R_0)$. Let $a_3$ be given by \eqref{a4-defn} and 
\begin{equation}\label{a5-defn-3}
a_4=\min\left(\frac{R_0}{4\sqrt{(n-1)\lambda}},\frac{a_3}{2}\right).
\end{equation}
Then $a_3>0$ and $a_4>0$.
Suppose there exists $r_2\in (r_1,R_0)$ such that $w(r_2)<a_4$. Let $(a,b)\in (0, R_0)$ be the maximal interval containing $r_2$ such that \eqref{w-upper-bd} holds. Then $a>r_1$ and $w(a)=a_4$. By \eqref{f-derivative-+ve}, $f(r)<0$ for all $0<r<R_0$. Hence by \eqref{w-upper-bd}, for any $a<r<b$ the right hand side of \eqref{w-derivative-eqn} is bounded below by
\begin{align*}\label{w_r-lower-bd4}
\ge&r(1+f_r(r)^2)\left(\frac{1}{\lambda w(r)}-\frac{4(n-1)w(r)}{R_0^2}\right)\notag\\
\ge&r(1+f_r(r)^2)\left(\frac{1}{4\lambda w(r)}\left(1-\frac{16(n-1)\lambda}{R_0^2}w(r)^2\right)+
\frac{3}{4\lambda w(r)}\right)\notag\\
\ge&\frac{3r_1}{4\lambda w(r)}.
\end{align*} 
Thus \eqref{w-lower-bd5} holds which contradicts  \eqref{w-upper-bd}. Hence no such $r_2$ exists and $w(r)\ge a_4$ for all $r_1\le r<R_0$ and \eqref{f-structure-ineqn5} holds with $\delta_2=a_4$ and the lemma follows.
\end{proof}

\begin{lem}\label{f''-positive-lemma}
Let $n\ge 2$, $\lambda>0$, $\mu<0$ and $R_0>0$. Suppose $f\in C^1([0,R_0))\cap C^2(0,R_0)$ is the solution of \eqref{imcf-graph-ode-initial-value-problem2} which satisfies \eqref{f-structure-ineqn2}. Then 
\begin{equation}\label{f''-positive}
f_{rr}(r)>0\quad\forall 0<r<R_0.
\end{equation}
\end{lem}
\begin{proof}
By \eqref{f-rr} there exists a constant $0<R_1<R_0$ such that
\begin{equation}\label{f''-local-positive}
f_{rr}(r)>0\quad\forall 0<r<R_1.
\end{equation}
Let $R_2=\max\{R\in (0,R_0):f_{rr}(r)>0\quad\forall 0<r<R\}$. Then $R_1\le R_2\le R_0$. Suppose $R_2<R_0$. Then 
\begin{equation}\label{f''-sign-eqn}
f_{rr}(R_2)=0, \qquad f_{rr}(r)>0\quad\forall 0<r<R_2\quad\mbox{ and }\quad f_{rrr}(R_2)\le 0.
\end{equation}
On the other hand by differentiating \eqref{imcf-graph-ode-initial-value-problem2} with respect to $r$ and putting $r=R_2$ we have
\begin{align*}
f_{rrr}(R_2)=&\frac{n-1}{R_2^2}(f_r(R_2)+f_r(R_2)^3)-\frac{n-1}{R_2}(f_{rr}(R_2)+3f_r(R_2)^2f_{rr}(R_2))\notag\\
&\qquad+\frac{1}{\lambda}\left\{\frac{4(1+f_r(R_2)^2)f_r(R_2)f_{rr}(R_2)}{R_2f_r(R_2)-f(R_2)}-\frac{R_2(1+f_r(R_2)^2)^2f_{rr}(R_2))}{(R_2f_r(R_2)-f(R_2))^2}\right\}\notag\\
=&\frac{n-1}{R_2^2}(f_r(R_2)+f_r(R_2)^3)\notag\\
>&0
\end{align*}
which contradicts \eqref{f''-sign-eqn}. Hence $R_2=R_0$ and the lemma follows.
\end{proof}

\begin{lem}\label{f-derivative-sequence-bd-lemma}
Let $n\ge 2$, $\lambda>\frac{1}{n-1}$, $\mu<0$ and $R_0>0$. Suppose $f\in C^1([0,R_0))\cap C^2(0,R_0)$ is the solution of \eqref{imcf-graph-ode-initial-value-problem2} which satisfies \eqref{f-structure-ineqn2}. Then there exists a constant $M_1>0$  such that
\begin{equation}\label{f-derivative-locally-finite10}
0\le f_r(r)\le M_1\quad\forall 0\le r<R_0.
\end{equation}
\end{lem}
\begin{proof}
Let $a_2$ be given by \eqref{a2-defn}. By Lemma \ref{f''-positive-lemma},
$a_3:=\lim_{r\to R_0}f_r(r)\in (0,\infty]$
exists. Suppose $a_3=\infty$.  We then claim that $a_2=\infty$. Suppose not. Then $a_2<\infty$ and $\mu<f(r)\le a_2$ for all $0<r<R_0$. By \eqref{imcf-graph-ode-initial-value-problem2},
\begin{align}\label{f''-f'3-ration-limit100}
\mbox{$\lim_{r\to\infty}$}\frac{f_{rr}(r)}{(1+f_r(r)^2)f_r(r)}=&\mbox{$\lim_{r\to\infty}$}\left(\frac{1}{\lambda}\cdot\frac{1+f_r(r)^2}{(rf_r(r)-f(r))f_r(r)}-\frac{n-1}{r}\right)\notag\\
=&\frac{1}{\lambda}\mbox{$\lim_{r\to\infty}$}\frac{f_r(r)^{-2}+1}{(r-(f(r)/f_r(r)))}-\frac{n-1}{R_0}\notag\\
=&\frac{1}{R_0}\left(\frac{1}{\lambda}-(n-1)\right)<0.
\end{align}
By \eqref{f''-f'3-ration-limit100} there exists $R_1\in (0,R_0)$ such that
\begin{equation*}
\frac{f_{rr}(r)}{(1+f_r(r)^2)f_r(r)}<0\quad\forall R_1\le r<R_0\quad\Rightarrow\quad 
f_{rr}(r)<0\quad\forall R_1\le r<R_0
\end{equation*}
which contradicts \eqref{f''-positive}. Hence $a_2=\infty$ and we can choose a constant $0<R_2<R_0$ such that $f(r)>0$ for any $R_2\le r<R_0$. 
We claim that there exists a constant $M_2>0$ such that
\begin{equation}\label{f'-f-ratio-bd}
f_r(r)\le M_2f(r)\quad\forall R_2\le r<R_0.
\end{equation}
Suppose \eqref{f'-f-ratio-bd} does not hold for any $M_2>0$. Then  there exists a sequence $\{r_k\}_{k=1}^{\infty}\subset (R_2, R_0)$, $r_k\to R_0$ as $k\to\infty$, such that
\begin{equation}\label{f'-f-ratio-limit-infty}
\mbox{$\lim_{r\to R_0}$}\frac{f_r(r_k)}{f(r_k)}=\infty.
\end{equation}
By \eqref{imcf-graph-ode-initial-value-problem2} and \eqref{f'-f-ratio-limit-infty},
\begin{align}\label{f''-f'3-ration-limit}
\mbox{$\lim_{k\to\infty}$}\frac{f_{rr}(r_k)}{(1+f_r(r_k)^2)f_r(r_k)}=&\mbox{$\lim_{k\to\infty}$}\left(\frac{1}{\lambda}\cdot\frac{1+f_r(r_k)^2}{(r_kf_r(r_k)-f(r_k))f_r(r_k)}-\frac{n-1}{r_k}\right)\notag\\
=&\frac{1}{\lambda}\mbox{$\lim_{k\to\infty}$}\frac{f_r(r_k)^{-2}+1}{(r_k-(f(r_k)/f_r(r_k)))}-\frac{n-1}{R_0}\notag\\
=&\frac{1}{R_0}\left(\frac{1}{\lambda}-(n-1)\right)<0.
\end{align}
By \eqref{f''-f'3-ration-limit} there exists $k_0\in\Z^+$ such that
\begin{equation*}
\frac{f_{rr}(r_k)}{(1+f_r(r_k)^2)f_r(r_k)}<0\quad\forall k\ge k_0\quad\Rightarrow\quad 
f_{rr}(r_k)<0\quad\forall k\ge k_0
\end{equation*}
which contradicts \eqref{f''-positive}. Hence there exists a constant $M_2>0$ such that \eqref{f'-f-ratio-bd} holds. Integrating \eqref{f'-f-ratio-bd} over $(R_2,R_0)$,
\begin{equation}\label{f-upper-bd11}
f(r)\le e^{M_2R_0}f(R_2)\quad\forall R_2\le r<R_0.
\end{equation}
By \eqref{f'-f-ratio-bd} and \eqref{f-upper-bd11},
\begin{equation*}
f_r(r)\le M_2e^{M_2R_0}f(R_2)\quad\forall R_2\le r<R_0
\end{equation*}
which contradicts the assumption that $a_3=\infty$. Hence $a_3<\infty$ and \eqref{f-derivative-locally-finite10} holds with $M_1=a_3$ and the lemma follows.
\end{proof}

We are now ready for the proof of Theorem \ref{existence_soln-thm}.

\noindent{\bf Proof of Theorem \ref{existence_soln-thm}}:
Since uniqueness of solution of  \eqref{imcf-graph-ode-initial-value-problem} follows by standard ODE theory. We only need to prove existence of solution of \eqref{imcf-graph-ode-initial-value-problem}.
By lemma \ref{local-existence-lem} there exists a constant $R_1>0$ such that the equation 
\eqref{imcf-graph-ode-initial-value-problem2} has a unique solution $f\in C^1([0,R_1))\cap C^2(0,R_1)$ which satisfies \eqref{f-structure-ineqn2} in $(0,R_1)$. Let $(0,R_0)$, $R_0\ge R_1$, be the maximal interval of existence of solution $f\in C^1([0,R_0))\cap C^2(0,R_0)$ of \eqref{imcf-graph-ode-initial-value-problem2} which satisfies \eqref{f-structure-ineqn2}.  

Suppose $R_0<\infty$.
By Lemma \ref{f-derivative-sequence-bd-lemma} there exists a constant $M_1>0$  such that \eqref{f-derivative-locally-finite10} holds.
By Lemma \ref{f-monotone-lemma} there exists a constant $\delta_2>0$ such that
\eqref{f-structure-ineqn5} holds. By \eqref{f-structure-ineqn2}, \eqref{f-derivative-+ve} and \eqref{f-derivative-locally-finite10},
\begin{equation}\label{f-rk-bd}
\mu<f(r)\le R_0M_1\quad\forall 0<r<R_0.
\end{equation}
By \eqref{f-derivative-+ve}, \eqref{f-structure-ineqn5}, \eqref{f-derivative-locally-finite10}, \eqref{f-rk-bd} and Lemma \ref{local-existence-extension-lem}, there exists a constant $\delta_1>0$ such that for any $r_1\in (R_0/2,R_0)$, there exists a unique solution $f_1\in C^2([r_1,r_1+\delta_1))$ of 
\eqref{imcf-graph-ode-bdary-value-problem} which satisfies
\eqref{f-structure-ineqn10} in $(r_1,r_1+\delta_1)$ with $a_0=f(r_1)$ and $b_0=f_r(r_1)$. We now choose $r_1\in (R_0/2,R_0)$ such that $R_0-r_1<\delta_1/2$. We extend $f$ to a function on $[0,r_1+\delta_1)$ by setting $f(r)=f_1(r)$ for all $r\in (r_1,r_1+\delta_1)$. Then $f$ is a solution of \eqref{imcf-graph-ode-initial-value-problem}  in $[0,r_1+\delta_1)$ which satisfies \eqref{f-structure-ineqn2} in $[0,r_1+\delta_1)$. Since $r_1+\delta_1>R_0$, this contradicts the choice of $R_0$. Hence $R_0=\infty$. By Lemma \ref{f-monotone-lemma}, \eqref{fr-+ve} holds and the theorem follows.

{\hfill$\square$\vspace{6pt}}

\section{Asymptotic behaviour of solution}
\setcounter{equation}{0}
\setcounter{thm}{0} 

In this section we will prove Theorem \ref{asymptotic-behaviour-time-infty-thm}. 
We first observe that by Lemma \ref{f''-positive-lemma} we have the following result.

\begin{cor}\label{f-to-infty-cor}
Let $n\ge 2$, $\lambda>\frac{1}{n-1}$, $\mu<0$ and $f$ be the unique solution of \eqref{imcf-graph-ode-initial-value-problem}  which satisfies \eqref{f-structure-ineqn}.
Then
\begin{equation}\label{f''-positive-10}
f_{rr}(r)>0\quad\forall r>0
\end{equation}
and 
\begin{equation}\label{f-tends-to-infty}
\mbox{$\lim_{r\to\infty}$}f(r)=\infty.
\end{equation}
\end{cor}

Note that by \eqref{f-tends-to-infty} there exists a constant $R_1>0$ such that
\begin{equation*}
f(r)>0\quad\forall r\ge R_1.
\end{equation*}

\begin{lem}\label{f'-to-infty-lem}
Let $n\ge 2$, $\lambda>\frac{1}{n-1}$, $\mu<0$ and $f$ be the unique solution of \eqref{imcf-graph-ode-initial-value-problem}  which satisfies \eqref{f-structure-ineqn}.
Then
\begin{equation}\label{f'-go-to-infty-at-x=infty}
\mbox{$\lim_{r\to\infty}$}f_r(r)=\infty.
\end{equation}
\end{lem}
\begin{proof}
By \eqref{f-structure-ineqn},
\begin{equation}\label{r-f'/f-lower-bd=1}
\frac{rf_r(r)}{f(r)}>1\quad\forall r\ge R_1.
\end{equation}
By \eqref{f''-positive-10},
$a_3:=\lim_{r\to\infty}f_r(r)\in (0,\infty]$
exists. Suppose $a_3<\infty$. Then by \eqref{f-tends-to-infty} and the l'Hospital rule,
\begin{equation}\label{r-f'/f-ratio-go-to1}
\mbox{$\lim_{r\to\infty}$}\frac{rf_r(r)}{f(r)}=\frac{\lim_{r\to\infty}f_r(r)}{\lim_{r\to\infty}\frac{f(r)}{r}}=\frac{\lim_{r\to\infty}f_r(r)}{\lim_{r\to\infty}f_r(r)}=\frac{a_3}{a_3}=1.
\end{equation}
Then by \eqref{imcf-graph-ode-initial-value-problem2}, \eqref{r-f'/f-lower-bd=1} and \eqref{r-f'/f-ratio-go-to1},
\begin{align*}
\mbox{$\lim_{r\to\infty}$}\frac{rf_{rr}}{(1+f_r^2)f_r}
=&\frac{1}{\lambda}\mbox{$\lim_{r\to\infty}$}\frac{r(1+f_r(r)^2)}{(rf_r(r)-f(r))f_r(r)}-(n-1)\notag\\
=&\frac{1}{\lambda}\mbox{$\lim_{r\to\infty}$}\frac{\frac{rf_r(r)}{f(r)}\cdot(1+f_r(r)^{-2})}{\frac{rf_r(r)}{f(r)}-1}-(n-1)\notag\\
=&\infty.
\end{align*} 
Hence there exists $R_2>R_1$ such that 
\begin{equation*}
\frac{rf_{rr}(r)}{(1+f_r(r)^2)f_r(r)}>1\quad\forall r\ge R_2.
\end{equation*}
Thus
\begin{equation*}
\frac{f_{rr}}{f_r}>\frac{1}{r}\quad\forall r\ge R_2.
\end{equation*}
Therefore
\begin{equation*}
f_r(r)\ge\frac{f_r(R_2)}{R_2}r\quad\forall r\ge R_2.
\end{equation*}
Hence
\begin{equation*}
a_3=\mbox{$\lim_{r\to\infty}$}f_r(r)=\infty
\end{equation*}
and contradiction arises. Hence $a_3<\infty$ does not hold. Thus $a_3=\infty$ and the lemma follows.
\end{proof}

\noindent{\bf Proof of Theorem \ref{asymptotic-behaviour-time-infty-thm}}: 
Let 
\begin{equation*}
q(r)=\frac{rf_r(r)}{f(r)}\quad\forall r\ge R_1.
\end{equation*}
By \eqref{imcf-graph-ode-initial-value-problem2} and a direct computation $q$ satisfies
\begin{equation}\label{q-eqn}
q_r(r)=\frac{q(r)}{r}\left\{(1+f_r(r)^2)\left(\frac{q(r)(1+f_r(r)^{-2})}{\lambda(q(r)-1)}-(n-1)\right)+1-q(r)\right\}\quad\forall r>R_1.
\end{equation}
Let $\alpha_0=\frac{\lambda (n-1)}{\lambda (n-1)-1}$, $0<\3<\min (1,\alpha_0-1)$,  $a_{1,\3}=\alpha_0+\3$ and $a_{2,\3}=\alpha_0-\3$. Then
$a_{1,\3}>\alpha_0>a_{2,\3}>1$ and
\begin{equation}\label{a1-epsilon-alpha-0-ineqn}
\frac{a_{1,\3}}{\lambda (a_{1,\3}-1)}<\frac{\alpha_0}{\lambda (\alpha_0-1)}=n-1<\frac{a_{2,\3}}{\lambda (a_{2,\3}-1)}.
\end{equation}
By \eqref{a1-epsilon-alpha-0-ineqn} there exists $M_1>1$ such that
\begin{equation*}
\delta_1:=\left(n-1-\frac{a_{1,\3}(1+M_1^{-2})}{\lambda (a_{1,\3}-1)}\right)(1+M_1^2)-1>0
\end{equation*}
and
\begin{equation*}
\delta_1':=(1+M_1^2)\left(\frac{a_{2,\3}}{\lambda (a_{2,\3}-1)}-(n-1)\right)-\alpha_0>0.
\end{equation*}
By \eqref{f'-go-to-infty-at-x=infty} there exists a costant $R_2>R_1$ such that 
\begin{equation}\label{f'-upper-bd5}
f_r(r)\ge M_1\quad\forall r\ge R_2.
\end{equation}
We will now prove that $q(r)$ is bounded above by $a_{1,\3}$ when $r$ is sufficiently large. Now either
\begin{equation}\label{q-upper-limit-bd}
q(r)\le a_{1,\3}\quad\forall r\ge R_2
\end{equation}
or 
\begin{equation}\label{q-big}
\exists r_1>R_2\quad\mbox{ such that }q(r_1)>a_{1,\3}
\end{equation}
holds. Suppose \eqref{q-big} holds.
Let $R_3=\sup\{r_2>r_1:q(r)>a_{1,\3}\quad\forall r_1\le r<r_2\}$. Suppose $R_3=\infty$. By \eqref{q-eqn} and \eqref{f'-upper-bd5}, $\forall r>r_1$,
\begin{align}\label{q'-ineqn1}
q_r\le&\frac{q(r)}{r}\left\{(1+f_r(r)^2)\left(\frac{a_{1,\3}(1+M_1^{-2})}{\lambda(a_{1,\3}-1)}-(n-1)\right)+1\right\}\notag\\
\le&\frac{q(r)}{r}\left\{-(1+M_1^2)\left(n-1-\frac{a_{1,\3}(1+M_1^{-2})}{\lambda(a_{1,\3}-1)}\right)+1\right\}\notag\\
\le&-\delta_1\frac{q(r)}{r}.
\end{align}
Hence
\begin{equation}\label{q'-ineqn10}
\frac{q_r}{q}\le-\frac{\delta_1}{r}\quad\forall r>r_1.
\end{equation}
Integrating \eqref{q'-ineqn10} over $(r_1,r)$,
\begin{equation*}
q(r)\le q(r_1)(r_1/r)^{\delta_1}\quad\forall r>r_1.
\end{equation*}
Hence
\begin{equation*}
q(r)<\frac{a_{1,\3}}{2}\quad\forall r>\left(\frac{a_{1,\3}}{2q(r_1)}\right)^{-1/\delta_1}r_1
\end{equation*}
which contradicts the assumption that $R_3=\infty$. Hence $R_3<\infty$ and by continuity of $q$, $q(R_3)=a_{1,\3}$. By \eqref{q'-ineqn1}, $q_r(R_3)\le-\delta_1q(R_3)/R_3<0$. Hence there a constant $\delta_2>0$ such that $q(r)<a_{1,\3}$ for all $R_3<r<R_3+\delta_2$. Let $R_4=\sup\{r_4>R_3:q(r)<a_{1,\3}\quad\forall R_3<r<r_4\}$. Suppose $R_4<\infty$. Then $q(R_4)=a_{1,\3}$ and $q_r(R_4)\ge 0$. On the other hand by an argument similar to the proof of \eqref{q'-ineqn1}, $q_r(R_4)\le -\delta_1q(R_4)/R_4<0$ and contradiction arises. Hence $R_4=\infty$. Thus
\begin{equation}\label{q-upper-limit-bd2}
q(r)\le a_{1,\3}\quad\forall r\ge R_3.
\end{equation}
By \eqref{q-upper-limit-bd} and \eqref{q-upper-limit-bd2} there  always exists some constant $R_5(\3)>R_2$ such that
\begin{equation}\label{q-upper-limit-bd3}
q(r)\le a_{1,\3}=\alpha_0+\3\quad\forall r\ge R_5(\3).
\end{equation}
We will now prove that $q(r)$ is bounded below by $a_{2,\3}$ when $r$ is sufficiently large. Now either
\begin{equation}\label{q-lower-limit-bd}
q(r)\ge a_{2,\3}\quad\forall r\ge R_5(\3)
\end{equation}
or 
\begin{equation}\label{q-small}
\exists r_1'>R_5(\3)\quad\mbox{ such that }q(r_1')<a_{2,\3}
\end{equation}
holds. Suppose \eqref{q-small} holds.
Let $R_3'=\sup\{r_2'>r_1':q(r)<a_{2,\3}\quad\forall r_1'<r<r_2'\}$. Suppose $R_3'=\infty$. By  \eqref{q-eqn} and \eqref{q-upper-limit-bd3}, $\forall r>r_1'$,
\begin{equation}\label{q'-ineqn2}
q_r\ge\frac{q(r)}{r}\left\{(1+M_1^2)\left(\frac{a_{2,\3}}{\lambda(a_{2,\3}-1)}-(n-1)\right)-\alpha_0\right\}
\ge\delta_1'\frac{q(r)}{r}.
\end{equation}
Hence
\begin{equation}\label{q'-ineqn11}
\frac{q_r}{q}\ge\frac{\delta_1'}{r}\quad\forall r>r_1'.
\end{equation}
Integrating \eqref{q'-ineqn11} over $(r_1',r)$,
\begin{equation*}
q(r)\ge q(r_1')(r/r_1')^{\delta_1'}\quad\forall r>r_1'.
\end{equation*}
Hence
\begin{equation*}
q(r)>2a_{2,\3}\quad\forall r>\left(\frac{2a_{2,\3}}{q(r_1')}\right)^{1/\delta_1'}r_1'
\end{equation*}
which contradicts the assumption that $R_3'=\infty$. Hence $R_3'<\infty$ and by continuity of $q$, $q(R_3')=a_{2,\3}$. By \eqref{q'-ineqn2}, $q_r(R_3')\ge\delta_1'q(R_3')/R_3'>0$. Hence there a constant $\delta_2'>0$ such that $q(r)>a_{2,\3}$ for all $R_3'<r<R_3'+\delta_2'$. Let $R_4'=\sup\{r_4>R_3':q(r)>a_{2,\3}\quad\forall R_3'<r<r_4\}$. Suppose $R_4'<\infty$. Then $q(R_4')=a_{2,\3}$ and $q_r(R_4')\le 0$. On the other hand by an argument similar to the proof of \eqref{q'-ineqn2}, $q_r(R_4')\ge\delta_1'q(R_4')/R_4'>0$ and contradiction arises. Hence $R_4'=\infty$. Thus
\begin{equation}\label{q-lower-limit-bd2}
q(r)\ge a_{2,\3}\quad\forall r\ge R_3'.
\end{equation}
By \eqref{q-lower-limit-bd} and \eqref{q-lower-limit-bd2} there  always exists some constant $R_5'(\3)>R_5(\3)$ such that
\begin{equation}\label{q-lower-limit-bd3}
q(r)\ge a_{2,\3}=\alpha_0-\3\quad\forall r\ge R_5'(\3).
\end{equation}
Since $\3\in (0,\min (1,\alpha_0-1))$ is arbitrary, by \eqref{q-upper-limit-bd3} and \eqref{q-lower-limit-bd3} we get \eqref{growth-rate}  and Theorem \ref{asymptotic-behaviour-time-infty-thm} follows.

{\hfill$\square$\vspace{6pt}}

\end{document}